\documentclass{lms}

\usepackage{amsmath,amssymb}

\def\ord{\mathop{\mathrm {ord}}\nolimits}

\newcommand{\Z}{\mathbf{Z}} 
\newcommand{\Q}{\mathbf{Q}}
\newcommand{\R}{\mathbf{R}}

\DeclareMathOperator{\inorm}{\mathbf N}

\newtheorem{theorem}{Theorem}
\newtheorem{algorithm}{Algorithm}
\newtheorem{definition}[theorem]{Definition}
\newtheorem{lemma}[theorem]{Lemma}
\newtheorem{proposition}[theorem]{Proposition}
\newtheorem{corollary}[theorem]{Corollary}

\begin{document}

\title{Computing in quotients of rings of integers}

\author{Claus Fieker and Tommy Hofmann}

\classno{11Y40 (primary), 11-04 (secondary)}

\extraline{Part of this work was supported through the DFG priority programm SPP 1489}

\maketitle 

\begin{abstract}
  We develop algorithms to turn quotients of rings of rings of integers into
  effective Euclidean rings by giving polynomial algorithms for
  all fundamental ring operations. In addition, we study normal forms for
  modules over such rings and their behavior under certain quotients.
  We illustrate the power of our
  ideas in a new modular normal form algorithm for modules over
  rings of integers, vastly outperforming classical algorithms.
\end{abstract}

\section{Introduction}
Rings of integers of number fields are fundamental rings in computational
number theory. Similar to algorithms over the integers, a common computational
tool is the transfer to quotient rings. This is for example done
to prevent intermediate coefficient explosion (Hermite form), allow
techniques based on the Chinese remainder theorem (CRT) (utilize the field
structure of suitable quotients) or limit the precision in $p$-adic
computations. For quotients $\mathbf Z/N\mathbf Z$, $N \neq 0$, of $\mathbf Z$, 
the rational integers, this
has a rich history, in particular normal forms for matrices or modules
over quotients have been studied extensively, both in their own right and
as a means to classify matrices and modules over $\mathbf Z$ itself.
An important observation was the fact that $\Z/N\Z$ can be given
the structure of a Euclidean ring thus allowing the use of general
algorithms designed for this class of rings.
In contrast to this situation, in rings of integers of number fields, the only
properties of quotient rings that have been exploited so far are the
fact that residue class rings of prime ideals are fields (CRT based algorithms)
and the obvious fact that quotients rings are finite, thus can be used to avoid 
intermediate coefficient swell (by reducing modulo some ideal every now and then).
However, the algorithms, e.g., the modular pseudo Hermite normal form of Cohen~\cite{Cohen1996}, or
Biasse--Fieker~\cite{Biasse2012} only add the reduction at crucial steps while still basically
maintaining the old, underlying, non-modular algorithm.

In this paper, we revive the fact that quotient rings of rings of integers are in fact
Euclidean rings allowing for efficient operations. As a result, over
such quotient rings, we can immediately use the rich history of algorithms 
for Euclidean rings. In particular, that allows a much
wider class of quotients to be used for non-trivial computations than
just the residue class fields. In fact, a short study will immediately show
that, since deterministic polynomial factorization over finite fields
is very slow, this gives rise to deterministic algorithms for the
computation of say determinants, of much better complexity.

We illustrate our new ideas by giving a new, truly modular, algorithm
for the computation of normal forms over rings of integers. Our algorithm,
by utilizing the Euclidean structure of suitable quotients, does
not need the complicated (and slow) operations of pseudomatrices
and ideals necessary in the classical approach. In fact, for random
matrices over rings of integers, the new algorithm has a much
better expected runtime than the $\mathbf Z$ algorithms on the 
corresponding $\mathbf Z$-module.

Starting with the Euclidean structure of quotient rings, we then study
matrix normal forms under projections before applying everything to
matrix normal forms over rings of integers.

\section{Background}\label{sec:background}

For the rest of the paper we fix an algebraic number field $K$ of degree $d$ with ring of integers $\mathcal O$.
If $\mathfrak m$ is a non-trivial ideal $\mathcal O$, we denote by $\inorm(\mathfrak m)$ the ideal norm of $\mathfrak m$, i.e., $\inorm(\mathfrak m) = \lvert \mathcal O /\mathfrak m \rvert$.
The main goal of this section is the description of the Euclidean structure of $(\mathcal O/\mathfrak m)$, where $\mathfrak m$ is a non-trivial ideal of $\mathcal O$, based on \cite{Fletcher1971}.
The first step consists of defining the Euclidean structure in case $\mathfrak m$ is a prime ideal power $\mathfrak p^l$, exploiting the special properties of the ring $(\mathcal O/\mathfrak p^l)$.
Finally a CRT based procedure is applied to obtain a Euclidean structure on the whole of $(\mathcal O/\mathfrak m)$ for arbitrary $\mathfrak m$.
\par
Recall that a commutative ring $R$ is called \textit{Euclidean} if there exists a function $\varphi \colon R\setminus \{ 0 \} \to \Z_{\geq 0}$ satisfying the following property: For all $a,b \in R, b \neq 0$ there exist $q,r \in R$ such that 
\begin{align}\label{division} a = q b + r \text{ with $\varphi(r) < \varphi(b)$ or $r = 0$.}
\end{align}
In this case $\varphi$ is called a \textit{Euclidean function} and (\ref{division}) is called \textit{Euclidean division}.
Note that this is not \textit{the} definition of Euclidean rings but one that suits our purpose. We refer the interested reader to \cite{Augargun1995} for an overview of possible definitions and relations between them.
\par
Beginning with a prime ideal power $\mathfrak p^l$ of $\mathcal O$, let us recall some facts about $(\mathcal O/\mathfrak p^l)$.
Let $\pi$ be an element of $\mathfrak p \setminus \mathfrak p^2$, the set of $\mathfrak p$-uniformizers.
Then $(\mathcal O/\mathfrak p^l)$ is a special principal ideal ring, i.e., a ring with unique maximal ideal which is nilpotent, and every ideal is of the form $(\overline \pi^k)$ with $0\leq k < l$.
\par
Fixing a set $S$ of coset representatives of $\mathcal O$ modulo $\mathfrak p$ it is well known that every element $\overline a$ of $(\mathcal O/\mathfrak p^l)$ can be uniquely written in the form
\[ \overline a = \sum_{i=v_\mathfrak p(a)}^{l-1} \overline s_i \overline \pi^i \]
with $s_i \in S$. Moreover $\overline a$ is invertible if and only if $s_0$ is a unit modulo $\mathfrak p$.
Using this representation it is easy to compute the cardinality of various objects.
\begin{lemma}\label{lem:idealsize}
  \begin{enumerate}
    \item 
      $\lvert (\mathcal O/\mathfrak p^l)^\times \rvert = \inorm(\mathfrak p)^{l-1}(\inorm(\mathfrak p)-1)$.
    \item
      $\lvert (\overline \pi^k) \rvert = \inorm(\mathfrak p)^{l-k}$ for $0 \leq k < l$.
    \item
      If $\mathfrak a$ is an ideal of $\mathcal O$, then $\overline{\mathfrak a} = (\overline \pi^{\min(v_\mathfrak p(\mathfrak a),l)})$ and $\lvert {\overline{\mathfrak a}} \rvert = \inorm(\mathfrak p)^{l-{\min(v_\mathfrak p(\mathfrak a),l)}}$.
    \item
      The number of generators of $(\overline \pi^k)$ is $\inorm(\mathfrak p)^{l-k-1}(\inorm(\mathfrak p) - 1)$ if $0 \leq k < l$ and $1$ if $k\geq l$.
  \end{enumerate}
\end{lemma}

By \cite[Proposition 7]{Fletcher1971} the function $(\mathcal O/\mathfrak p^l) \setminus \{ \overline 0 \} \rightarrow \Z_{\geq 0}, \, \overline a \mapsto v_\mathfrak p(a)$ 
defines a Euclidean function on $(\mathcal O/\mathfrak p^l)$.
For the sake of completeness we sketch the argument:
The above representation of elements of $(\mathcal O/\mathfrak p^l)$ shows that every element $\overline a$ can be written as $u_a \overline \pi^k$ for some unit $u_a$ and unique integer $k$ (in fact $k = v_\mathfrak p(a)$).
If $\overline a$ and $\overline b$ are elements of $(\mathcal O/\mathfrak p^l)$ with $\overline b \neq \overline 0$, then
\[ \overline a = \begin{cases} \overline 0 \cdot \overline b + \overline a, & \quad \text{if }v_\mathfrak p(a) < v_\mathfrak p(b), \\
    \overline u_a \overline u_b^{-1} \overline \pi^{v_\mathfrak p(a) - v_\mathfrak p(b)} \cdot \overline b + \overline 0, & \quad \text{if }v_\mathfrak p(a) \geq v_\mathfrak p(b).\end{cases} \]
is a Euclidean division.
Composing this Euclidean function with the monotone increasing function $x \mapsto \inorm(\mathfrak p)^x$ yields
\[ \varphi_\mathfrak p \colon (\mathcal O/\mathfrak p^l) \setminus \{ \overline 0 \} \longrightarrow \Z_{\geq 0}, \, \overline a \longmapsto \inorm (\mathfrak p)^{v_\mathfrak p(a)}, \]
also turning $(\mathcal O/\mathfrak p^l)$ into a Euclidean ring.
Moreover we extend the function to the whole of $(\mathcal O/\mathfrak p^l)$ by setting $\varphi_\mathfrak p(\overline 0) = \inorm(\mathfrak p)^{l}$, such that $\varphi_\mathfrak p(\overline a) = \inorm(\mathfrak p)^{\min(v_\mathfrak p(a),l)}$ for all $\overline a \in (\mathcal O / \mathfrak p^l)$.
\par
Now we can put everything together.
For each prime divisor $\mathfrak p$ of $\mathfrak m$ denote by $\varphi_\mathfrak p \colon (\mathcal O / \mathfrak p^{v_\mathfrak p(\mathfrak m)}) \to \Z$ the Euclidean function defined in the previous paragraph and by $\overline a_\mathfrak p \in (\mathcal O /\mathfrak p^{v_\mathfrak p(\mathfrak m)})$ the $\mathfrak p$-component of an element $\overline a \in (\mathcal O/\mathfrak m)$ under the natural isomorphism $(\mathcal O/\mathfrak m) \cong \prod_{\mathfrak p\mid \mathfrak m} (\mathcal O/\mathfrak p^{v_\mathfrak p(\mathfrak m)})$.
\begin{proposition}
  The ring $(\mathcal O/\mathfrak m)$ together with 
  \[ \varphi \colon (\mathcal O/\mathfrak m) \setminus \{ \overline 0 \} \longrightarrow \Z_{\geq 0}, \, \overline a \longmapsto \inorm(a,\mathfrak m) \] 
is a Euclidean ring. 
\end{proposition}
\begin{proof}
  The proof of \cite[Proposition 6]{Fletcher1971} shows that $(\mathcal O/\mathfrak m)$ is a Euclidean ring with Euclidean function $\sum_\mathfrak p \varphi_\mathfrak p(\overline a_\mathfrak p)$. 
  But it is easy to see that the proof remains valid if the sum is replaced by $f((\varphi_\mathfrak p(\overline a_\mathfrak p))_\mathfrak p)$, where $f \colon \prod_{\mathfrak p \mid \mathfrak m}\R  \to \R$ is any monotonic multivariate function.
  The result then follows by choosing $f$ to be the product and noting that $\inorm(a,\mathfrak m) = \varphi(\overline a) = \prod_\mathfrak p \varphi_\mathfrak p(\overline a_\mathfrak p)$.
\end{proof}

We end this section with some remarks on division in $(\mathcal O/\mathfrak m)$.
First note that due to the presence of zero-divisors the division in $(\mathcal O/\mathfrak m)$ is not unique. 
To illustrate the occurring pitfalls we consider an example in $\Z / 30 \Z$.
It is easy to see that $\overline a = \overline 6$ and $\overline b = \overline{10}$ satisfy $(\overline a , \overline b ) = (\overline g)$ with $g = \overline 2$. 
This shows that $\overline g$ is a greatest common divisor of $\overline a$ and $\overline b$.
We now want to divide by $\overline g$: While the equations $\overline g \cdot \overline{18} = \overline a$ and $\overline g \cdot \overline{20} = \overline b$ show that $\overline{18}$ and $\overline{20}$ are valid quotients, they are not coprime in $\Z/30\Z$ as $(\overline{18},\overline{20}) = (\overline 2)$.
This is in total contrast to the situation of integral domains, where dividing by a greatest common divisor produces coprime elements. 
Here we can try to find coprime quotients by choosing different ones.
Now $\overline g \cdot \overline 3 = \overline a$ and $\overline g \cdot 5 = \overline b$ show that $\overline 3$ and $\overline 5$ will also do and they are fortunately coprime in $\Z/30\Z$.
\par
We now prove that this is always possible by choosing the quotients as small as possible with respect to the Euclidean function.

\begin{proposition}\label{prop:div}
  Let $\overline a,\overline b \in (\mathcal O/\mathfrak m)$. Then the following holds:
  \begin{enumerate}
    \item
      The element $\overline b$ divides $\overline a$ if and only if $(a,\mathfrak m)(b,\mathfrak m)^{-1}$ is an integral ideal.
    \item
      An element $\overline c \in (\mathcal O/\mathfrak m)$ satisfies $\overline b \overline c = \overline a$ if and only if $(c,\mathfrak m) \subseteq (a,\mathfrak m)(b,\mathfrak m)^{-1}$.
    \item
      If $\overline c \in (\mathcal O/\mathfrak m)$ satisfies $\overline b \overline c = \overline a$, then $\varphi(\overline a)/\varphi(\overline b)$ divides $\varphi(\overline c)$.
    \item
      Let $\overline c \in (\mathcal O/\mathfrak m)$ such that $\overline b \overline c= \overline a$. Then $\varphi(\overline a) /\varphi(\overline b) = \varphi(\overline c)$ is equivalent to $(\overline c) = \overline{(a,\mathfrak m)(b,\mathfrak m)^{-1}}$.
    \item
      Let $\overline g \in (\mathcal O/\mathfrak m)$ be a greatest common divisor of $\overline a ,\overline b$, i.e., $(\overline g) = (\overline a,\overline b)$. Assume that $\overline e,\overline f$ are elements of $(\mathcal O/\mathfrak m)$ such that $\overline e \overline g = \overline a$, $\overline f \overline g = \overline b$, $\varphi(\overline e) = \varphi(\overline a)/\varphi(\overline g)$ and $\varphi(\overline f) = \varphi(\overline b)/\varphi(\overline g)$. Then $\overline e$ and $\overline f$ are coprime, i.e., $(\overline e,\overline f) = (\mathcal O/\mathfrak m)$.
  \end{enumerate}
\end{proposition}

\begin{proof}
  (i): This follows from the fact the $\overline b \mid \overline a$ is equivalent to $\overline b_\mathfrak p \mid \overline a_\mathfrak p$ for all prime divisors $\mathfrak p$ of $\mathfrak m$.
  \par
  (ii): For each prime divisor $\mathfrak p$ of $\mathfrak m$ we have $\overline b_\mathfrak p \overline c_\mathfrak p = \overline a_\mathfrak p$. If $\overline a_\mathfrak p \neq 0$ (and therefore $\overline b_\mathfrak p\neq 0$) this is equivalent to $v_\mathfrak p(c) = v_\mathfrak p(a) - v_\mathfrak p(b) = v_\mathfrak p((a,\mathfrak m)(b,\mathfrak m)^{-1})$.
  If $\overline a_\mathfrak p = \overline b_\mathfrak p = 0$ then this is equivalent to $v_\mathfrak p(c) \geq 0 = v_\mathfrak p((a,\mathfrak m)(b,\mathfrak m)^{-1})$.
  If $\overline a_\mathfrak p = 0$ and $\overline b_\mathfrak p \neq 0$, then this is equivalent to $v_\mathfrak p(c) \geq v_\mathfrak p(\mathfrak m) - v_\mathfrak p(b) = v_\mathfrak p((a,\mathfrak m)(b,\mathfrak m)^{-1})$.
  Now the claim follows.
  \par
  (iii) and (iv): This follows from (ii).
  \par
  (v): Note that $(g,\mathfrak m) = (a,b,\mathfrak m)$.
  By (ii) the assumption on the Euclidean function implies $(e,\mathfrak m) = (a,\mathfrak m)(a,b,\mathfrak m)^{-1}$ and $(f,\mathfrak m) = (b,\mathfrak m)(a,b,\mathfrak m)^{-1}$.
  From this one deduces that $(e,f,\mathfrak m ) = \mathcal O$, i.e., $(\overline e,\overline f) = (\mathcal O/\mathfrak m)$.
\end{proof}

\section{Basic operations}

In order to describe the complexity of our algorithms we will rely on a modified notion of basic operations introduced by Mulders and Storjohann in \cite{Storjohann1998}. Let $(R,\varphi)$ be a Euclidean ring and $a,b \in R$. Then a \textit{basic operation} is one of the following:
\begin{enumerate}
  \item[(B1)]
    For $\ast \in \{ +, - , \cdot\}$ return $a \ast b$.
  \item[(B2)]
    If $b$ divides $a$ in $R$ return an element $\mathsf{div}(a,b) = c \in R$ such that $bc = a$.
  \item[(B3)]
    If $b \neq 0$ return $\mathsf{eudiv}(a,b) = (q,r) \in R^2$ such that $a = qb+r$ with $r=0$ or $\varphi(r) < \varphi(b)$.
  \item[(B4)]
    Return $\mathsf{xgcd}(a,b) = (g,s,t,u,v) \in R^5$ such that $(g) = (a,b)$, $g = sa + tb$, $ua+vb= 0$ and $sv - ut = 1$, i.e.,
    \[ \begin{pmatrix} g & 0 \end{pmatrix}  = \begin{pmatrix} a & b \end{pmatrix} \begin{pmatrix} s & u \\ t & v \end{pmatrix} \]
    and the transformation matrix is unimodular.
  \item[(B5)]
    Return $\mathsf{Ann}(a) = c$ such that $(c) = \mathrm{Ann}(a) = \{ r \in R \, | \, ra = 0\}$.
\end{enumerate}
Note that in \cite{Storjohann1998} it is shown that in case of $R = \Z/N\Z$ operations (B1) through (B5) can be performed using $O(M(\log(N) \log(\log(N))))$ bit operations, where $M(t)$ is a bound on the number of bit operations required to multiply two $\lceil t \rceil$-bit integers.
\par
We now turn to the case $R = (\mathcal O/\mathfrak m)$, for which there exists an additional basic operation.
\begin{enumerate}
  \item[(B6)] Given an integral ideal $\mathfrak a$ of $\mathcal O$, return an element $\mathsf{gen}(\mathfrak a) = \overline c \in (\mathcal O/\mathfrak m)$ such that $\overline {\mathfrak a} = (\overline c)$ in $(\mathcal O/\mathfrak m)$.
\end{enumerate}
We now want to show how each basic operation (Bi) in $(\mathcal O/\mathfrak m)$, $1 \leq i \leq 6$, can be solved algorithmically using basic operation in $\Z/N\Z$, where $N = \inorm(\mathfrak m)$ is the norm of $\mathfrak m$.
We assume that we are given $\Z$-bases $(\omega_i)_{1\leq i \leq d}$ and $(\nu_i)_{1 \leq i \leq d}$ of $\mathcal O$ and $\mathfrak m$ respectively such that $\nu_i = n_i \omega_i$ with integers $n_i \in \Z_{\geq 1}$, $1 \leq i \leq d$, i.e., the basis matrix of $\mathfrak m$ is diagonal.
Then the map
\[ (\mathcal O/\mathfrak m) \longrightarrow (\Z/n_1\Z) \times \dotsb \times (\Z/n_d \Z), \, \overline{\sum_i a_i \omega_i} \longmapsto (\overline a_1,\dotsc,\overline a_d) \]
is an isomorphism of abelian groups which we use to identify $(\mathcal O/\mathfrak m)$ with $\prod_i \Z/n_i \Z$.
\par
Evaluating the canonical map $\mathcal O \to (\mathcal O/\mathfrak m)$ at an element $\sum_i a_i \omega_i$ consists of $d$ divisions with remainder and the addition of two elements in $(\mathcal O/\mathfrak m)$ consists of $d$ additions in $\Z/n_i \Z$.
As the above map is not multiplicative, multiplication of two elements $\overline a = (\overline a_1,\dotsc,\overline a_d)$, $\overline b = (\overline b_1,\dotsc,\overline b_d) \in (\mathcal O/\mathfrak m)$ is more involved.
More precisely the element $\overline c = (\overline c_1,\dotsc,\overline c_d) \in (\mathcal O/\mathfrak m)$ with $\overline a \overline b = \overline c$ is given by
\[ \overline c_k = \overline{\sum_i \sum_j a_i b_j \Gamma_{i,j}^k}  \in (\Z/n_k \Z), \]
where $(\Gamma_{i,j}^k)_{i,j,k}$ denotes the structure constants of the $\Z$-algebra $\mathcal O$ with respect to the basis $(\omega_i)_{1\leq i \leq d}$.
Thus for each $1 \leq k \leq d$ we need $d^2$ basic operations in $(\Z/n_k \Z)$ to compute $\overline c_k$.
\par
To accomplish (B2), denote by $M_b \in \Z^{d \times d}$ the representation matrix of $\mathcal O \to \mathcal O, x \mapsto bx$ with respect to $(\omega_i)$, where each entry is reduced modulo $N$, and by $M_\mathfrak m$ the diagonal basis matrix of $\mathfrak m$.
Then $\overline a =\overline b \overline c$ for some element $c \in (\mathcal O/\mathfrak m)$ if and only if the equation $(M_b|M_\mathfrak m)X = a$ is solvable.
As this linear system can be solved modulo $N$, we need $O(d^3)$ basic operations in $\Z/N\Z$.
Note that the kernel of this matrix is (the lift) of $\mathrm{Ann}(\overline b)$, the annihilator of $\overline b$ in $(\mathcal O/\mathfrak m)$.
\par
So far we have shown that operations (B1) and (B2) can be performed using $O(d^3)$ basic operations in $\Z/N\Z$ (for the sake of simplicity a basic operation in $\Z/k\Z$ with $1 \leq k \leq N$ is counted as a basic operation in $\Z/N\Z$).
\par
We now turn to the more involved operations (Bi), $3 \leq i \leq 6$, the big difference to (B1) being the non-uniqueness of the operations (again mainly due to the presence of zero-divisors).
Using the Chinese remainder theorem we will see that the defining properties of the operations can be stated purely in terms of valuations at each prime ideal dividing $\mathfrak m$.
Therefore the main task will be the construction of integral elements with prescribed behavior at a finite set of prime ideals.
While there exist deterministic algorithms for these kind of problems, they have the major flaw that they need a costly prime ideal factorization of $\mathfrak m$.
To overcome this difficulty, in this article we will pursue the idea of probabilistic algorithms. 
More precisely our algorithms will be of Las Vegas type with expected polynomial running time, which can be easily turned into Monte Carlo algorithms if wished.
The running time of our algorithms will depend on the value
\[ p_\mathfrak m = \frac{\lvert (\mathcal O /\mathfrak m)^\times\rvert }{\lvert (\mathcal O/\mathfrak m)\rvert} = \prod_{\mathfrak p \mid \mathfrak m} \left( 1 - \frac{1}{\inorm(\mathfrak p)} \right). \]
In Section~\ref{sec:splitting} we will discuss the size of $p_\mathfrak m$ and the applicability of the presented algorithms.
\par
We assume that we have access to an oracle producing random elements in any finite ring of the form $\Z/ k\Z$, $k \in \Z_{> 0}$.
During the complexity analysis we will omit the costs of calling this oracle.

\subsection{Euclidean function and division with remainder}

\begin{lemma}
  Let $\overline a \in (\mathcal O/\mathfrak m)$. Computing $\varphi(\overline a)$ can be done using $O(d^3)$ basic operations in $\Z/N\Z$.
\end{lemma}
\begin{proof}
  We first compute the $d$ products $\overline a \overline \omega_i$ for $1 \leq i \leq d$ using $O(d^3)$ basic operations in $\Z/N\Z$.
  Denoting by $\gamma_1,\dotsc,\gamma_d$ the canonical lifts of these elements we know that
  $\gamma_1,\dotsc,\gamma_d,\nu_1,\dotsc,\nu_d$ constitute a $\Z$-generating system of $(a) + \mathfrak m$.
  Computing the Hermite normal form basis of this generating system then can be done using $O(d^3)$ basic operations in $\Z/N \Z$ while the norm computation takes $O(d)$ such operations.
\end{proof}

\begin{algorithm}[(Probabilistic Euclidean division)]\label{alg:div}
  Let $\overline a, \overline b \in (\mathcal O/\mathfrak m)$, $\overline b \neq \overline 0$. The following steps return $\mathsf{eucdiv}(\overline a,\overline b)$.
  \begin{enumerate}
    \item
      Choose $\overline q \in (\mathcal O/\mathfrak m)$ uniformly distributed and compute $\overline r = \overline a - \overline q \overline b$.
    \item
      If $\varphi(\overline r) \geq \varphi(\overline a)$ go to Step~(i).
    \item
      Return $(\overline q, \overline r)$.
  \end{enumerate}
\end{algorithm}

\begin{lemma}
  Let $\overline a,\overline b \in (\mathcal O/\mathfrak m)$ such that $\overline b$ does not divide $\overline a$. For each prime divisor $\mathfrak p$ of $\mathfrak m$ define
  \[  S_\mathfrak p = \begin{cases} (\mathcal O/\mathfrak p^{v_\mathfrak p(\mathfrak m)}),&\text{if $0 < v_\mathfrak p(a) < v_\mathfrak p(b)$},\\
      (\mathcal O/\mathfrak p^{v_\mathfrak p(\mathfrak m)})^\times, &\text{if $v_\mathfrak p(b) < v_\mathfrak p(a)$},\\
      \{ \overline x \in (\mathcal O/\mathfrak p^{v_\mathfrak p(\mathfrak m)} \, \mid \, \inorm( (a+xb) ,\mathfrak p^{v_\mathfrak p(\mathfrak m)})  \leq \inorm(b,\mathfrak p^{v_\mathfrak p(\mathfrak m)} ) \}, &\text{if $v_\mathfrak p(a) = v_\mathfrak p(b)$}.
                  \end{cases} \]
  Then the following holds:
  \begin{enumerate}
    \item
      If $\overline c \in (\mathcal O/\mathfrak m)$ is an element such that $\overline c_\mathfrak p \in S_\mathfrak p$ for all prime divisors $\mathfrak p$ of $\mathfrak m$, then $\varphi(\overline a+\overline b \overline c) < \varphi(\overline b)$.
    \item
      We have $\{ \overline c \in (\mathcal O/\mathfrak m) \, \mid \, \varphi(\overline a+ \overline b\overline c) < \varphi(\overline b) \} \geq \lvert (\mathcal O/\mathfrak m)^\times \rvert$.
    \item
      If $\overline q \in (\mathcal O/\mathfrak m)$ is uniformly distributed in $(\mathcal O/\mathfrak m)$, then the probability that $\overline a = \overline q \overline b + (\overline a - \overline q \overline b)$ is a Euclidean division is at least $p_\mathfrak m$.
 \end{enumerate}
\end{lemma}

\begin{proof}
  (i): Let $\overline c_\mathfrak p \in S_\mathfrak p$.
  In the second and third case we have $v_\mathfrak p(a+bc) \leq v_\mathfrak p(b)$ while in the first case we have $v_\mathfrak p(a+bc) = v_\mathfrak p(a) < v_\mathfrak p(b)$.
  Since $\overline b$ does not divide $\overline a$ there exists a prime divisor $\mathfrak p$ of $\mathfrak m$ such that $0 < v_\mathfrak p(a) < v_\mathfrak p(b)$ implying that $\inorm((a+bc),\mathfrak p^{v_\mathfrak p(\mathfrak m)}) < \inorm(b , \mathfrak p^{v_\mathfrak p(\mathfrak m)})$.
  Thus we have $\varphi(\overline a + \overline b \overline c) < \varphi(\overline b)$.\par
  (ii): It remains to show $\lvert S_\mathfrak p\rvert \geq \lvert (\mathcal O/\mathfrak p^{v_\mathfrak p(\mathfrak m)}) ^\times \rvert$ in the case $v_\mathfrak p(a) = v_\mathfrak p(b)$.
  If $v_\mathfrak p(b) \geq v_\mathfrak p(\mathfrak m)$, then $S_\mathfrak p = \mathcal O/\mathfrak p^{v_\mathfrak p(\mathfrak m)}$ and we are done.
  Therefore let $v_\mathfrak p(b) < v_\mathfrak p(\mathfrak m)$ and consider the natural map $\pi \colon (\mathcal O/\mathfrak p^{v_\mathfrak p(\mathfrak m)}) \to (\mathcal O/\mathfrak p^{v_\mathfrak p(b) + 1})$.
  The set $\pi(S_\mathfrak p)$ is the complement of the set of solutions $\overline a = -\overline b \overline x$ with $\overline x \in (\mathcal O/\mathfrak p^{v_\mathfrak p(b)+ 1})$.
  As this equation has $\inorm((b),\mathfrak p^{v_\mathfrak p(b)+1}) = \inorm(\mathfrak p^{v_\mathfrak p(b)})$ solutions we have $\lvert \pi(S_\mathfrak p)\rvert = \inorm(\mathfrak p^{v_\mathfrak p(b) + 1}) - \inorm(\mathfrak p^{v_\mathfrak p(b)})$.
  It follows that $\lvert S_\mathfrak p\rvert = \inorm(\mathfrak p)^{v_\mathfrak p(\mathfrak m) - (v_\mathfrak p(b)+1)} \lvert \pi(S_\mathfrak p) \rvert = \lvert (\mathcal O/\mathfrak p^{v_\mathfrak p(\mathfrak m)})^\times \rvert$.
  \par
  (iii): This follows from (ii).
\end{proof}

\begin{proposition}
  Algorithm~\ref{alg:div} is correct and the expected number of basic operations in $\Z/N\Z$ is $O((1/p_\mathfrak m) d^3))$.
\end{proposition}

\begin{proof}
  We need to count the expected number of repetitions of Step~(i).
  It is easy to see that for $i \in \Z_{\geq 1}$, with probability $p_\mathfrak m(1-p_\mathfrak m)^{i-1}$ the number of repetitions of Step~(i) is $i$. 
  Thus the expected number is $p_\mathfrak m \sum_{i=1}^\infty i (1-p_\mathfrak m)^{i-1} = p_\mathfrak m ( 1/ p_\mathfrak m + (1-p_\mathfrak m)/{p_\mathfrak m^2}) =1/p_\mathfrak m$. 
  Now the claim follows as Step~(i) needs $O(d^3)$ basic operations in $\Z/N\Z$.
\end{proof}

\subsection{Finding a generator of an ideal and computing the annihilator}

Let $\mathfrak a$ be an ideal of $\mathcal O$.
It is easy to see that for an element $c \in \mathcal O$ the equation $(\overline c) = \overline{\mathfrak a}$ holds if and only if for all prime divisors $\mathfrak p$ of $\mathfrak m$ we have $v_\mathfrak p(c) = v_\mathfrak p(\mathfrak a,\mathfrak m)$.

\begin{algorithm}\label{alg:prob_gen}
  Let $\mathfrak a$ be an integral of $\mathcal O$.
  The following steps return $\overline c \in (\mathcal O/\mathfrak m)$ such that $(\overline c) = \overline {\mathfrak a}$.
  \begin{enumerate}
    \item
      Compute $(\mathfrak a,\mathfrak m)$.
    \item
      Choose $\overline c \in (\mathfrak a,\mathfrak m)/(N^2)$ uniformly distributed.
    \item
      If $(\mathfrak a,\mathfrak m) \neq (N,c)$ go to Step (ii).
    \item
      Return $\overline c \in (\mathcal O/\mathfrak m)$.
  \end{enumerate}
\end{algorithm}

\begin{lemma}
  Algorithm~\ref{alg:prob_gen} is correct and the expected number of basic operations in $\Z/N\Z$ is $O((1/p_\mathfrak m)d^3)$.
\end{lemma}

\begin{proof}
  We prove the following: If $\mathfrak a$ is an integral ideal of $\mathcal O$ and $\overline c$ is chosen uniformly in $(\mathfrak a,\mathfrak m)/(N^2)$, then the probability that $(\mathfrak a,\mathfrak m) = (N,c)$ is $p_\mathfrak m$.
  Let $\mathfrak b = (\mathfrak a,\mathfrak m)$ and fix one prime divisor $\mathfrak p$ of $\mathfrak m$.
  We want to count the elements $\overline c \in \mathfrak b /(N^2)$ such that $v_\mathfrak p(c) = v_\mathfrak p(\mathfrak b)$.
  Note that $v_\mathfrak p(N^2) > v_\mathfrak p(\mathfrak b)$ and therefore $c \in \mathfrak b \setminus \mathfrak b \mathfrak p$ is equivalent to $\overline c \in \mathfrak b/(N^2)\setminus \mathfrak b \mathfrak p/(N^2)$. Counting the elements in these sets we see that probability that an element $\overline c \in \mathfrak b /(N^2)$ satisfies $v_\mathfrak p(c) = v_\mathfrak p(\mathfrak a)$ is $(1 - 1/\inorm(\mathfrak p))$.
  \par
  Note that Step~(i) needs $O(d^3)$ basic operations in $\Z/N\Z$.
  We have already shown that the expected number of executions of Step~(iii) is $1/p_\mathfrak m$.
  As each execution consists of $O(d^3)$ basic operations in $\Z/N\Z$, the claim follows.
\end{proof}

\begin{lemma}
  Let $\overline b \in (\mathcal O/\mathfrak m)$. Then we can compute $\overline c = \mathsf{Ann}(\overline b)$ with an expected number of $O((1/p_\mathfrak m)d^3)$ basic operations in $\Z/N\Z$.
\end{lemma}

\begin{proof}
  After computing the annihilator as the kernel of $M_b$ modulo $N$ (as for (B2)) using $O(d^3)$ basic operations, we apply Algorithm~\ref{alg:prob_gen} to obtain a generator.
\end{proof}

\subsection{Extended GCD computation}

We now turn to the $\mathsf{xgcd}$ problem.
In case of the rational integers $\Z$ the task is easy:
If $g$ is a greatest common divisor of two integers $a,b \in \Z$ we can compute $s,t \in \Z$ such that $g = s a + t b$. Then
\[ \begin{pmatrix} g & 0 \end{pmatrix} = \begin{pmatrix} a & b \end{pmatrix} \begin{pmatrix} s & -t/g \\ t & s/g \end{pmatrix} \]
and we are done. While we can of course just use the normal Euclidean
algorithm to find the cofactors, this is, in our case, rather expensive as each
Euclidean division requires a random search. On the other hand, computing the
GCD directly using ideals takes only \textit{one} random search.
\par
As the underlying idea is that dividing by a greatest common divisor produces coprime elements, the example at the end of Section~\ref{sec:background} shows that we cannot blindly adapt this in the presence of zero-divisors.
Fortunately Proposition~\ref{prop:div} shows that there exists minimal quotients $\overline e,\overline f$ with respect to the Euclidean function such that $\overline e \overline g =\overline a$, $\overline f \overline g =\overline b$ and $(\overline e,\overline f) = (\mathcal O/\mathfrak m)$.
In particular there exists $\overline u,\overline v \in (\mathcal O/\mathfrak m)$ such that $\overline e \overline u + \overline f \overline v = 1$. 
A quick calculation shows that
\[ \begin{pmatrix} \overline g & \overline 0 \end{pmatrix} = \begin{pmatrix} \overline a & \overline b \end{pmatrix} \begin{pmatrix} \overline u & - \overline f \\ \overline v & \overline e \end{pmatrix} \]
is a unimodular transformation implying that $\mathsf{xgcd}(\overline a,\overline b) = (\overline g, \overline u,\overline v, - \overline f ,\overline e)$ is valid.
\par
In order to apply this we need to explain how to find minimal quotients and how to express a greatest common divisor as a linear combination.

\begin{lemma}
  \begin{enumerate}
    \item
    Let $\overline b$ be a divisor of $\overline a$. An element $\overline c \in (\mathcal O/\mathfrak m)$ with $\overline c \overline b = \overline a$ and $\varphi(\overline c) = \varphi(\overline a)/\varphi(\overline b)$ can be computed using an expected number of $O((1/p_\mathfrak m)d^3)$ basic operations in $\Z/N\Z$.
  \item
    Let $\overline e,\overline f \in (\mathcal O/\mathfrak m)$ be such that $(\overline e,\overline f) = (\mathcal O/\mathfrak m)$. Then $\overline u,\overline v$ with $\overline u \overline e + \overline v \overline f = 1$ can be computed using $O(d^3)$ basic operations in $\Z/N\Z$.
  \item
    Let $\overline a,\overline b \in (\mathcal O/\mathfrak m)$. Then $\mathsf{xgcd}(\overline a,\overline b)$ can be computed with an expected number of $O((1/p_\mathfrak m)d^3)$ basic operations.
\end{enumerate}
\end{lemma}

\begin{proof}
  (i): Using (B2) we can compute a fixed quotient $\overline c_0$.
  Moreover we have seen that at the same time we obtain a basis of an ideal $\mathfrak a$ of $\mathcal O$ with $\overline {\mathfrak a} = \mathrm{Ann}(\overline b)$. 
  Invoking (B6) we can compute a generator of the ideal $\overline{\mathfrak a}$.
  Now we choose uniformly distributed elements $\overline q \in \overline{\mathfrak a}$ until $\varphi(\overline c_0 + \overline q) = \varphi(\overline a)/\varphi( \overline b)$.
  If this is the case then $\overline c_0 + \overline q$ is a quotient which is minimal with respect to the Euclidean function. 
  Proposition~\ref{prop:div} shows that if $\overline q$ is uniformly distributed in $\mathrm{Ann}(\overline b)$, then $\overline c_0 + \overline q$ is uniformly distributed in $\overline{(a,\mathfrak m)(b,\mathfrak m)^{-1}}$.
  Now the claim follows from Lemma~\ref{lem:idealsize}.
  \par
  (ii): As in the case of division, we see that the set of tuples $(\overline x,\overline y) \in (\mathcal O/\mathfrak m)^2$ with $\overline x \overline e + \overline y \overline f = \overline 1$ is the set of solutions of a $d \times 3d$ matrix with entries in $\Z$.
  As in addition this system can be solved modulo $N$, the task of finding a suitable tuple $(\overline x,\overline y)$ can be solved using $O(d^3)$ basic operations in $\Z/N\Z$.
  \par
  (iii): Follows from (ii) and (iii).
\end{proof}

\begin{corollary}
  Any basic operation in $(\mathcal O/\mathfrak m)$ can be performed with an expected number of $O((1/p_\mathfrak m)d^3)$ basic operations in $\Z/N\Z$.
\end{corollary}

\section{Applications to matrix normal forms}

When working with algebraic number fields the objects of desire often carry the structure of finitely generated torsion-free modules over $\Z$.
While the structure theorem for modules over $\Z$ asserts the freeness of such modules, the Hermite normal form (HNF) and algorithms for computing it bring them fully under control.
They not only allow for the computation of a basis given a generating set, but they also enable us to solve various algorithmic problems.
\par
Based on the extended GCD, it is straight forward to formulate a naive algorithm for computing the HNF over $\Z$.
Unfortunately, as in the case of Gau\ss ian elimination over $\Q$, coefficient swell occurs.
Although there are various techniques to handle this circumstance, the most natural one is the use of residual methods, which goes back to Iliopoulos \cite{Iliopoulos1989} and Domich, Kannan and Trotter \cite{Domich1987}: Instead of computing the HNF over $\Z$, one computes a normal form over $\Z/d\Z$ for some $d \in \Z$ and lifts the result back to $\Z$.
If $d$ is chosen to be a multiple of the determinant of the lattice spanned by the rows of the matrix, this will yield a correct result.
\par
The aim of this section is to introduce residual methods to the computation of normal forms of $\mathcal O$-modules by passing to a quotient ring $(\mathcal O/\mathfrak m)$ for some suitable integral ideal $\mathfrak m$ and by lifting the result back to $\mathcal O$.

\subsection{Strong echelon form for principal ideal rings}

Given a ring $R$ and a matrix $A \in R^{n\times m}$ denote by $S(A) \subseteq R^m$ the row span of $A$.
The idea of attaching a unique matrix normal form to submodules of $R^m$, where $R$ is a principal ideal ring, goes back to Howell \cite{Howell1986}. 
He introduced a normal form (now called the \textit{Howell normal form}) of submodules of $(\Z/d\Z)^m$ and an algorithm for computing it, such that two modules are equal if and only if their Howell normal forms coincide.
In his PhD thesis Storjohann \cite{Storjohann2000} has generalized this notion to arbitrary principal ideal rings.
\par
In this article we will adapt the Howell normal form to our needs.
For an $R$-module $M \subseteq R^m$ and $1 \leq i \leq m$ we define $S_i(M)$ to be the set of all elements of $M$ with last $i$ entries zero.
For convenience we set $S_i(A) = S_i(S(A))$ if $A$ is matrix over $R$ with $m$ columns.

\begin{definition}
  Let $M \subseteq R^m$ be an $R$-module. A matrix $H = (h_{ij}) \in R^{n \times m}$, $n \geq m$, is called \textit{strong echelon form} of $M$ if and only if
  \begin{enumerate}
    \item[(S1)] For $1 \leq i \leq m$ the $i$-th row of $H$ is zero or $i = \max\{ 1 \leq j \leq m \, | \, h_{ij} \neq 0 \}$. For $i > m$ the $i$-th row of $H$ is zero.
    \item[(S2)] For $1 \leq i \leq m$ the rows $1,\dotsc,i$ generate $S_{m-i}(M)$.
  \end{enumerate}
\end{definition}

To illustrate the definitions consider the following matrices over $\Z/6\Z$:
\[ A = \begin{pmatrix} \overline 0 & \overline 0 \\ \overline 1 & \overline 3 \end{pmatrix}, \quad
   B = \begin{pmatrix} \overline 2 & \overline 0 \\ \overline 5 & \overline 3 \end{pmatrix}, \quad
   C = \begin{pmatrix} \overline 0 & \overline 0 \\ \overline 2 & \overline 0 \\ \overline 5 & \overline 3 \end{pmatrix}. \quad
   D = \begin{pmatrix} \overline 2 & \overline 0 \\ \overline 5 & \overline 3 \\ \overline 0 & \overline 0  \end{pmatrix}. \quad
\]
It is easy to see that they have the same span.
While the matrix $A$ has a minimal number of non-zero rows the element $(\overline 2 ,\overline 0) \in S(A)$ shows that $A$ does not satisfy (S2).
On the other hand the matrix $C$ violates (S1).
Thus only $B$ and $D$ are strong echelon forms of $M$.
\par
A few words on the relation between the strong echelon form and the Howell normal form: 
In contrast to the Howell normal form we ``order'' the basis elements.
This will be important in Section~\ref{subsec:comb} where we describe the combination of strong echelon forms.
Note that we will use the strong echelon form over $(\mathcal O/\mathfrak m)$ only as an auxiliary step to obtain normal forms over $\mathcal O$.
Since this does not require the strong echelon form to be unique, this explains the absence of appropriate restrictions in the definition.
For working with $(\mathcal O/\mathfrak m)$-modules themselves we can recover uniqueness easily by the following steps.
We have to show how to find a fixed representative modulo $(\mathcal O/\mathfrak m)^\times$ and modulo $(\overline d)$ for some $\overline d \in (\mathcal O/\mathfrak m)$.
The former problem can be solved by noting that if $\overline a$ is an element of $(\mathcal O/\mathfrak m)$, then the coset of $\overline a$ modulo $(\mathcal O/\mathfrak m)^\times$ is equal to the set of all $\overline b \in (\mathcal O/\mathfrak m)$ with $(b,\mathfrak m) = (a,\mathfrak m)$.
Thus by choosing a generator of $\overline{(a,\mathfrak m)}$ in a deterministic way we obtain a unique representative.
By reducing the off-diagonal elements modulo the unique HNF basis of $(d,\mathfrak m)$, where $d$ is the corresponding diagonal entry, we obtain unique representatives for the off-diagonal elements.
\par
Based on Howell's approach Storjohann and Mulders describe in \cite{Storjohann1998} a simple algorithm for computing the Howell normal form over $\Z/d\Z$, which easily generalizes to any ring supporting basic operations (Bi), $1 \leq i \leq 6$.
The following modified version yields a strong echelon form.

\begin{algorithm}[(Strong echelon form over principal ideal rings)]\label{alg:strongechelonform}
  Let $A \in R^{n\times m}$ be a matrix with $n \geq m$. The following steps return a strong echelon form of $A$.
  \begin{enumerate}
    \item
      (This puts $A$ into triangular form).
      For $1 \leq i < j \leq n$ compute $(g,s,t,u,v) = \mathsf{xgcd}(a_{j,i},a_{j,j})$ and set
      \[ \left(\begin{matrix} A_j \\ A_i \end{matrix}\right) = \left( \begin{matrix} s & t \\ u & v \end{matrix} \right) \left( \begin{matrix} A_j \\ A_i \end{matrix} \right). \]
    \item
      Augment $A$ with one zero row.
    \item
      For $1 \leq j \leq m$ do the following:
      \begin{enumerate}
        \item
          If $a_{j,j} \neq 0$ compute $c = \mathsf{Ann}(a_{j,j})$ and set $A_{n+1} = c A_j$. If $a_{j,j} = 0$ then set $A_{n+1} = A_j$.
        \item
          For $j+1 \leq i \leq m$ compute $(g,s,t,u,v) = \mathsf{xgcd}(a_{i,i},a_{n+1,i})$ and set
          \[ \left(\begin{matrix} A_i \\ A_{n+1} \end{matrix}\right) = \left( \begin{matrix} s & t \\ u & v \end{matrix} \right) \left( \begin{matrix} A_i \\ A_{n+1} \end{matrix} \right). \]
      \end{enumerate}
    \item Sort the rows such that (S1) is satisfied.
    \item Return $A$.
\end{enumerate}
\end{algorithm}

\subsection{Modular computation of a strong echelon form}\label{subsec:comb}

One of the reasons we have introduced the strong echelon form (instead of using the equally unknown Howell normal form) is the important fact, that it allows for efficient residual computations.
To be more precise let $R$ be a principal ideal ring and $a,b,e,f \in R$ elements such that $ab = 0$ and $1= ea + fb$.
Denote by $\pi_a$ and $\pi_b$ the canonical projections of $R$ onto $R/(a)$ and $R/(b)$ respectively.
By abuse of notation we denote the induced projections $R^m \to (R/(a))^m$ and $R^{n\times m} \to (R/(a))^{n\times m}$ also by $\pi_a$; we do the same for $\pi_b$.
Then for any $R$-module $M \subseteq R^m$ the equation 
\begin{align}  M = 1M = eaM + fbM = ea(M + b R^n) + fb(M + aR^n) \label{eq:dec} \end{align}
holds.
As $M+ aR^n = \pi_a^{-1}(\pi_a(M))$ and $M + bR^n = \pi_b^{-1}(\pi_b(M))$ we see that $M$ can be obtained by lifting the modules $\pi_a(M)$ and $\pi_b(M)$, which are now living over the (hopefully ``smaller'') rings $R/(a)$ and $R/(b)$, back to $R$. 
The following lemma shows that by using the strong echelon form the lifting procedure comes for free.

\begin{lemma}\label{lem:howcrt1}
  Assume that $A \in R^{n \times m}$ is a matrix such that $\pi_a(A) \in (R/(a))^{n\times m}$ is a strong echelon form of $\pi_a(M)$ and every non-zero diagonal Element of $A$ is a divisor of $a$. Then $A$ is a strong echelon form of $M + aR^n$.
\end{lemma}

\begin{proof}
  Given $v \in S_j(M + aR^n)$ we want to show that $v \in S(A)$.
  We prove the statement by induction on $j$.
  If $\pi_a(v_j) = 0$, then $v_j$ is a multiple of $a$.
  In particular there exists $r \in R$ such that $v - r A_j \in S_{j+1}(M + aR^n)$.
  If $\pi_a(v_j) \neq 0$ then property~(S1) implies that there exists $r_i \in R$ with $v - \sum_{i=1}^j r_i A_i \in S_{j}(aR^n)$.
  Thus again there exists $r \in R$ such that $v = \sum_{i=1}^j r_i A_i - r A_j \in S_{j+1}(M+aR^n)$.
  This implies $M + a R^m = S(A)$ and at the same time we have shown that $A$ satisfies property~(S1) and (S2).
\end{proof}

Thus by computing strong echelon forms over $R/(a)$ and $R/(b)$ we can compute strong echelon forms of $M + a R^n$ and $M + bR^n$. 
We now turn to the recombination step.
Let $A$ and $B$ be strong echelon forms of $M+ aR^n$ and $M+ bR^n$ respectively.
By padding $A$ or $B$ with zero rows we may assume that $A$ and $B$ have the same number of rows.

\begin{lemma}\label{lem:howcrt2}
   The matrix $fb A + eaB$ is a strong echelon form of $M$.
\end{lemma}

\begin{proof}
  Firstly we show $M = S(fb A + eaB)$. 
  Equation~(\ref{eq:dec}) implies that $M$ is generated by $fbA_i,eaB_i$, $1 \leq i\leq n$.
  Therefore it is sufficient to prove $fb A_i, eaB_i \in S(fbA + eaB)$.
  As $fb$ is an idempotent, i.e., $(fb)^2 = fb$. we have $fb A_i = (fb)^2 A_i + (fb)(ea)B_i = fb(fb A_i + eaB_i) \in S(fbA + eaB)$ and analogously $eaB_i \in S(fbA +eaB)$.
  \par
  Sine $eaB$ and $fbA$ have property~(S1), so does the sum.
  Property~(S2) follows by decomposing an element $v \in M$ into $v = fbv + eav$ and applying property~(S2) of $eaB$ and $fbA$.
\end{proof}

Now let $\mathfrak m$ and $\mathfrak n$ be coprime integral ideals of $\mathcal O$. 
We want to apply the preceding discussion to the computation of a strong echelon form of an $(\mathcal O/\mathfrak m \mathfrak n)$-module $M$.
Denote by $\overline a$ and $\overline b$ generators of the ideals $\overline {\mathfrak m}$ and $\overline{\mathfrak n}$ in $(\mathcal O/\mathfrak m \mathfrak n)$. 
Then $\overline a \overline b = 0$, and $(\mathcal O/\mathfrak m \mathfrak n)/(\overline a)$ and $(\mathcal O/\mathfrak m \mathfrak n)/(\overline b)$ are isomorphic to $\mathcal O/\mathfrak m$ and $\mathcal O/\mathfrak n$ respectively.
We have canonical projections $\pi_a = \pi_\mathfrak m \colon (\mathcal O/\mathfrak m\mathfrak n) \to (\mathcal O/\mathfrak m)$ and $\pi_b = \pi_\mathfrak n \colon (\mathcal O/\mathfrak m\mathfrak n) \to (\mathcal O/\mathfrak n)$.
As $\overline a$ and $\overline b$ are coprime, we can compute $\overline e,\overline f \in (\mathcal O/\mathfrak m \mathfrak n)$ such that $\overline e \overline a + \overline f \overline b = 1$.
Thus we are in a situation where we can apply Lemma~\ref{lem:howcrt1} and~\ref{lem:howcrt2}.
The only missing step is the normalization of the diagonal elements in the assumption of Lemma~\ref{lem:howcrt2}.
\par
We assume that $A'$ is a matrix over $(\mathcal O/\mathfrak m\mathfrak n)$ such that $\pi_\mathfrak m(A')$ is a strong echelon form of $\pi_\mathfrak m(M)$.
We define a new matrix $A$ over $(\mathcal O/\mathfrak m\mathfrak n)$ by setting the $i$-th row $A_i$ to be
\[ A_i = \overline b A_i' + (\overline a \delta_{i,j})_{1 \leq j \leq n} \]
for $1 \leq i \leq n$, where $\delta_{i,j}$ denotes the Kronecker delta.
As $\overline b$ is a unit modulo $\mathfrak m$ and $\pi_\mathfrak m(\overline a) = 0$, the matrix $\pi_\mathfrak m(A)$ is also a strong echelon form of $\pi_\mathfrak m(M)$.
We claim that $A$ satisfies the assumption of Lemma~\ref{lem:howcrt1}.
To prove this we show that for all $\overline d \in (\mathcal O/\mathfrak m \mathfrak n)$ the element $\overline b \overline d + \overline a$ is a divisor of $\overline a$ in $(\mathcal O/\mathfrak m \mathfrak n)$.
Note that this is equivalent to $\min(v_\mathfrak p(bd +a),v_\mathfrak p(\mathfrak m \mathfrak n)) \leq \min(v_\mathfrak p(a),v_\mathfrak p(\mathfrak m \mathfrak n))$ for all prime divisor $\mathfrak p$ of $\mathfrak m \mathfrak n$.
If $\overline d = 0$ this holds obviously.
Therefore we may assume $\overline d \neq 0$.
But then the claim follows easily by noting that $v_\mathfrak p(a) = v_\mathfrak p(\mathfrak m)$ if $\mathfrak p \mid \mathfrak m$ and $v_\mathfrak p(b) > 0 = v_\mathfrak p(a)$ if $\mathfrak p \mid \mathfrak n$.

\subsection{Normal forms for modules over $\mathcal O$}

Since $\mathcal O$ is in general not a principal ideal domain, finitely generated torsion-free modules over $\mathcal O$ are not necessarily free.
For this reason the connection between such modules and matrix normal forms is more subtle then in the principal ideal domain case.
While for any $\mathcal O$-module $M \subseteq \mathcal O^m$ there exists some matrix $A \in \mathcal O^{m\times n}$ such that $S(A) = M$, we cannot expect to find a triangular shaped matrix with this property.
For if this is the case, $M$ is the direct sum of the rows of $A$ and therefore free over $\mathcal O$.
\par
Although $\mathcal O$ is not a principal ideal domain, the properties of being a Dedekind ring are strong enough to prove a weakened classification theorem of $\mathcal O$-modules.
More precisely Steinitz~\cite{Steinitz1911,Steinitz1912} has shown that there exists fractional ideals $\mathfrak a_1,\dotsc,\mathfrak a_n$ of $K$ and a matrix $A \in K^{n \times m}$ with rows $A_1,\dotsc,A_n$ such that $M = \mathfrak a_1 A_1 \oplus \dotsb \oplus \mathfrak a_n A_n$.
To work with these objects, Cohen \cite{Cohen1996} has introduced the notion of a \textit{pseudomatrix}, which is just a pair $((\mathfrak a_i)_{1\leq i \leq n},A)$ consisting of a family of fractional ideals of $K$ (the \textit{coefficient ideals}) and a matrix $A \in K^{n \times m}$.
If $\mathcal P = ((\mathfrak a_i)_i,A)$ is such a pseudomatrix, we define $S(\mathcal P)$ to be $\sum_{i=1}^n \mathfrak a_i A_i$, the \textit{span} of the pseudomatrix $\mathcal P$.
In case $\sum \mathfrak a_i A_i = \bigoplus \mathfrak a_i A_i$ we call $\mathcal P$ a \textit{nice} pseudomatrix.
Note that $\mathcal P$ is nice if $A$ is of triangular shape.
\par
The problem of computing a nice pseudomatrix goes back to Bosma and Pohst \cite{Bosma1991}.
Based on similar ideas, Cohen introduced in \cite{Cohen1996} the notion of pseudo Hermite normal form  (pseudo-HNF) of a module---similar to the HNF over principal ideal domains---and described an algorithm for computing it.
To be more precise, a pseudomatrix $\mathcal P = ((\mathfrak a_i),A)$ with span $M$ is called \textit{a pseudo-HNF of $M$}, if $A$ is a lower triangular matrix with $1$ being the last non-zero element in each non-zero row.
By choosing the off-diagonal elements in fixed sets of coset representatives, the pseudo-HNF of an $\mathcal O$-module is unique.
Recently, Biasse and Fieker \cite{Biasse2012} have modified Cohen's algorithm to formulate a provable polynomial time algorithm for computing the pseudo-HNF.

\subsection{From $\mathcal O$ to $(\mathcal O/\mathfrak m)$ to $\mathcal O$}

Let $\mathcal P = ((\mathfrak a_i),A)$ be a pseudomatrix.
So far the underlying idea of all known algorithms for computing a nice pseudomatrix of $S(\mathcal P)$  is to transform $A$ into triangular shape, while carefully adjusting the coefficient ideals ensuring that the span does not change.
The necessary modifications of the coefficient ideals are the heart and at the same time the bottleneck of these algorithms.
In \cite{Biasse2012} even costly lattice reduction algorithms are necessary to bound the size of the objects during the algorithm and to ensure polynomial time complexity.
\par
We now describe how most of the ideal arithmetic can be avoided by passing to a suitable quotient ring of $\mathcal O$.
From now on we assume that the span $M = S(\mathcal P)$ is an $\mathcal O$-module of rank $m$ contained in $\mathcal O^m$ and $A \in K^{n \times m}$ with $n \geq m$.
As in the integer case the key idea is that there exists an integral ideal $\mathfrak m$ of $\mathcal O$ such that $\mathfrak m \mathcal O^m \subseteq M$.
Denote by $\pi_\mathfrak m$ the canonical projection $\mathcal O \to \mathcal O / \mathfrak m$ and the induced projections on $\mathcal O^m$ and $\mathcal O^{n\times m}$.

\begin{algorithm}
  The following steps return a matrix $\overline B \in (\mathcal O/\mathfrak m)^{n\times m}$ such that $S(\overline B) = \pi_\mathfrak m(S(\mathcal P))$.
  \begin{enumerate}
    \item
      For $1 \leq i \leq n$ find elements $a_i \in K$ such that $\mathfrak b_i = a_i \mathfrak a_i$ is integral and coprime to $\mathfrak m$, and divide row $A_i$ by $a_i$.
    \item
      For $1 \leq i, j \leq m$ write $A_{ij} = a_{ij}/b_{ij}$ with $a_{ij},b_{ij} \in \mathcal O$.
    \item
      return $\overline B = ( \overline a_{ij} \overline b_{ij}^{-1})_{i,j}$.
  \end{enumerate}
\end{algorithm}

A few remarks on the correctness. Step~(i) does not change the span and the new coefficient ideals $\mathfrak b_i$---being coprime to $\mathfrak m$---satisfy $\pi_\mathfrak m(\mathfrak b_i) = (\mathcal O/\mathfrak m)$.
Moreover the relation $\mathfrak m \mathcal O^m \subseteq M \subseteq \mathcal O^m$ implies the the denominator of all matrix entries are coprime to $\mathfrak m$ and thus invertible modulo $\mathfrak m$.
Finding the elements $a_i$ in Step~(i) is just another application of the approximation theorem (see \cite[Corollary 1.3.9]{Cohen2000}) and can therefore be performed using Belabas' algorithm.
\par
Applying Algorithm~\ref{alg:strongechelonform} to the matrix $\overline B$ obtained in the preceding algorithm we arrive---after removing zero rows---at a matrix $C \in \mathcal O^{m\times m}$ such that $\pi_\mathfrak m(C)$ is a strong echelon form of $\pi_\mathfrak m(M)$.
The connection to the original module $M$ is given by the following lemma.

\begin{lemma}\label{lem:lem1}
  Assume that $C \in \mathcal O^{m\times m}$ is a matrix such that $\pi_\mathfrak m(C)$ is a strong echelon form of $\pi_\mathfrak m(M)$.
  Then the pseudomatrix $\mathcal P' = (I,C)$ with $I = (\mathcal O,\dotsc,\mathcal O,\mathfrak m,\dotsc,\mathfrak m)$ and $D = (C^t|\mathbf{I}_m^t)^t$ satisfies $S(\mathcal P') = M$.
\end{lemma}

\begin{proof}
  Let $v$ be an element of $M$.
  As $\pi_\mathfrak m(v) \in \pi_\mathfrak m(M) = S(\pi_\mathfrak m(B))$, there exists $a_i \in \mathcal O$ such that $v - \sum a_i C_i \in \mathfrak m \mathcal O^m$.
  Now the claim follows.
\end{proof}

Thus by computing a preimage $C = (c_{ij})$ of a strong echelon form over the ring $(\mathcal O/\mathfrak m)$, we arrive at the following pseudomatrix spanning the original module (we write the coefficient ideals in front of the corresponding rows):

\newcommand\bigzero{\makebox(0,0){\text{\huge0}}}
\newcommand{\bord}[1]{\multicolumn{1}{c|}{ #1 }}
\makeatletter
\def\vcdots{\vbox{\baselineskip2\p@ \lineskiplimit\z@
        \kern2\p@\hbox{.}\hbox{.}\hbox{.}\kern2\p@}}
\makeatother

\begin{equation}
  \mathcal P' = \begin{array}{c} \mathcal O \\ \mathcal O \\ \mathcal O \\ \mathcal O \\ \mathcal O \\ \mathfrak m \\ \mathfrak m \\ \mathfrak m \\ \mathfrak m \\ \mathfrak m \end{array}
  \left(
    \begin{array}{ccccc}\cline{1-1}
      \bord {c_{1,1}}    &        &    &    &  \\ \cline{2-2}
    
      \ast & \bord {c_{2,2}}      &     & \bigzero    &  \\ 
    \ast & \ast    & \dots    &     &  \\ \cline{4-4}
    \ast      & \ast & \dots & \bord {\dots}   &  \\ \cline{5-5}
    \ast      &  \ast  & \dotsb     & \ast & \bord {c_{m,m}} \\ \cline{1-5}
    1 & & & & \\ 
    & 1 & & \bigzero & \\ 
    &  & \dots & & \\
     & \bigzero & & \dots & \\
     & & & & 1 \\ 
  \end{array}\right).
\end{equation}

We now apply the classical pseudo-HNF algorithm of Cohen to this pseudomatrix.
The special shape allows us to skip most of the steps and we actually never have to work with all of $\mathcal P'$.

\begin{algorithm}[(Demodularization)]\label{alg:demod}
  Let $C \in \mathcal O^{m\times m}$ be a matrix such that $\pi_\mathfrak m(C)$ is a strong echelon form of $\pi_\mathfrak m(M)$.
  The following steps return a pseudo-HNF with span equal to $M$.
  \begin{enumerate}
    \item
      For $i=m,\dotsc,1$ do the following:
    \item
      Let $\mathfrak g = (c_{i,i},\mathfrak m)$ and compute $x \in (c_{i,i})\mathfrak g^{-1}$, $y \in \mathfrak m \mathfrak g^{-1}$ such that $1 = x+ y$.
    \item
      Set $\mathfrak b_i = \mathfrak g$, $B_i = x A_i / c_{i,i}$ and $B_{i,i} = 1$.
    \item
      return $((\mathfrak b_i)_{1\leq i \leq m}, B)$.
  \end{enumerate}
\end{algorithm}

\begin{theorem}
  Algorithm~\ref{alg:demod} is correct.
\end{theorem}

\begin{proof}
  For the proof it is convenient to think of all operations applied to the pseudomatrix $\mathcal P'$ in (5.2), which actually spans the module $M$ by Lemma~\ref{lem:lem1}.
  We now take a look at Step~(ii) and Step~(iii).
  For the sake of convenience we consider only the case $i=m$.
  By \cite[Prop. 1.3]{Cohen1996} the pseudomatrices
  \begin{equation*} 
    \begin{array}{c} (c_{m,m}) \\ \mathfrak m \end{array}
  \left(
    \begin{array}{cccc}
      {c_{m,1}}/{c_{m,m}}    &   \dotsc    & c_{m,{m-1}}/c_{m,m}   & 1  \\
        0                   & \dotsb & 0 & 1 \\
  \end{array}\right).
\end{equation*}
and
\begin{equation*}
  \begin{array}{c} \mathfrak g \\ \mathfrak m \mathfrak g^{-1} \end{array}
  \left(
    \begin{array}{cccc}
      x ({c_{m,1}}/{c_{m,m}})    &   \dotsc    & x (c_{m,{m-1}}/c_{m,m})   & 1  \\
      - c_{m,1} & \dotsc & -c_{m,m-1} & 0 \\
  \end{array}\right).
\end{equation*}
span the same module. 
We need to show that the second row of the latter pseudomatrix is superfluous.
Let $v$ be in the span of the second row.
In particular $v \in S(M)$ and $\pi_\mathfrak m(v) \in \pi_\mathfrak m(M) = S(\pi_\mathfrak m(C))$.
As the last entry is zero we have $\pi_\mathfrak m(v) \in S_1(\pi_\mathfrak m(C))$.
As $\pi_\mathfrak m(C)$ is a strong echelon form this implies that there exists $r_j \in \mathcal O$ such that $v - \sum_{j=1}^{m-1} r_j C_j \in S_1(\mathfrak m\mathcal O^m)$.
Thus $v = \sum_{j=1}^{m-1}r_j C_j + \sum_{j=1}^{m-1} s_j e_j$ for some $s_j \in \mathfrak m$ and $e_j = (\delta_{ji})_{1 \leq i \leq m}$.
\end{proof}

A few remarks on the complexity.
While the inversion of ideals requires at most $O(d^3)$ operations using a precomputed $2$-element representation of the codifferent, the multiplication requires $O(d^4)$ operations if both ideals are given by their $\Z$-bases.
Therefore a naive approach to Step~(ii) requires $O(d^4)$ operations.
But we can do better by noting that 
\[ \mathfrak m\mathfrak g^{-1} = (\mathfrak m (a)^{-1} \cap \mathcal O) \text{ and } (a)\mathfrak g^{-1} = (\mathfrak m(a)^{-1} \cap \mathcal O)^{-1} \cap \mathcal O. \]
Now the ideal product involves a principal ideal and can be performed using at most $O(d^3)$ operations.
Since the artificially introduced inversions and intersections with $\mathcal O$ require at most $O(d^3)$ operations, the whole step requires at most $O(d^3)$ operations.
Note that the naive application of the pseudo-HNF algorithm of Cohen would have required $O(n^2)$ operations similar to Step~(ii) involving growing ideals. Let us summarize our algorithm.
\begin{algorithm}\label{alg:sum}
  Given an $\mathcal O$-module $M$ and a pseudomatrix $\mathcal P$ with $S(\mathcal P) = M$, the following steps return a pseudo-HNF of $M$.
  \begin{enumerate}
    \item
      Find an ideal $\mathfrak m$ such that $\mathfrak m\mathcal O^m \subseteq M$ (see Section~\ref{sec:modulus}).
    \item
      Compute $C \in \mathcal O^{m\times m}$ such that $\pi_\mathfrak m(C)$ is a strong echelon form of $\pi_\mathfrak m(M)$ using Algorithm~\ref{alg:strongechelonform} and Algorithm~\ref{alg:zsplit}
    \item
      Return the result of Algorithm~\ref{alg:demod} applied to $C$.
  \end{enumerate}
\end{algorithm}

Let $\mathcal P = ((\mathfrak a_i),A)$ be a pseudomatrix with $A \in K^{n\times m}$ and span $M \subseteq \mathcal O^m$.
Note that in order for the modular algorithm to be applicable, it is crucial that there exists some integral ideal $\mathfrak m$ such that $\mathfrak m \mathcal O^m \subseteq M \subseteq \mathcal O^m$, which is equivalent to $A$ being of rank $m$.
As in the case $\mathcal O = \Z$ without this assumption this modular technique won't work.
\par
Now assume that $\mathcal H = ((\mathfrak b_i)_i,H)$ is a pseudo-HNF of $\mathcal P$.
A transformation matrix from $\mathcal P$ to $\mathcal H$ is a matrix $U \in \operatorname{GL}_n(K)$ with $u_{ij} \in \mathfrak b_i \mathfrak a_j^{-1}$, $1 \leq i, j \leq n$, and $UA = H$.
We note that our algorithm for computing a pseudo-HNF does not produce such a transformation.
This is unsurprising, as the same problem can also be observed in case of modular $\Z$-HNF algorithms, see for example~\cite{Hafner1991}.
In our algorithm, the problems already show up during the calculation over the quotient ring, since our strong echelon form algorithm does not compute a transformation matrix either.
If needed, we can recover a transformation matrix $U$ from $\mathcal P$ and $\mathcal H$ by solving linear systems of equations over $K$ and by computing the kernel of $A$.
The problem of computing a transformation matrix efficiently during the modular algorithm is open.

It is worthwhile to mention the special case $\mathcal O = \Z$, for which we can recover the classical HNF over $\Z$.
Let $M \subseteq \Z^m$ be a $\Z$-module of rank $m$ with basis matrix $A \in \Z^{m \times m}$.
Moreover let $d \in \Z_{>0}$ be an element with $d \Z^{m} \subseteq M$ and $C \in \Z^{m \times m}$ such that $C$ modulo $d\Z$ is a strong echelon form of $\pi_d(M) \subseteq (\Z/d\Z)^m$.
Note that by multiplying the rows of $C \bmod d\Z$ with suitable elements of $(\Z/d\Z)^\times$ and by adding suitable elements, we can achieve that the diagonal elements of $C$ actually divide $d$.
Thus the whole demodularization step is superfluous and $C$ is the HNF of $M$.
This is in total contrast to the classical modular HNF algorithms, where after a computation in $\Z/d\Z$ one has to compute again a non-modular HNF of a matrix similar to (5.2) (see \cite[Section 2.1]{Hafner1991}).

\subsection{Finding a modulus $\mathfrak m$}\label{sec:modulus}

The crucial step in our normal form algorithm is the existence of an integral ideal $\mathfrak m$ with $\mathfrak m\mathcal O^m \subseteq M$. 
While there are situations in which such an $\mathfrak m$ is readily available, for example when working with ideals in relative extensions of number fields, let us briefly sketch how to obtain such an $\mathfrak m$ in general.
\par
First assume that $\mathcal P = ((\mathfrak a_i),A)$ is a pseudomatrix with $A \in K^{m \times m}$ and span equal to $M$. 
Then it is well known that the ideal $\mathfrak d = \det(A)\cdot \mathfrak a_1 \dotsm \mathfrak a_m \subseteq \mathcal O$ has the property that $\mathfrak d \mathcal O^m \subseteq M$.
Therefore it remains to show how to compute $\det(A)$ efficiently.
By clearing denominators we may assume that $A$ has only integral coefficients.
Now a small primes modular algorithm can be used: Find enough rational primes $p_i$ such that $\det(A)$ can be recovered from the determinant of $A$ modulo $(\prod_{i} p_i)$.
For each prime number compute the determinant of $A$ modulo $(p_i)$ using unimodular triangulation (Step~(i) of Algorithm~\ref{alg:strongechelonform}).
Now use the Chinese remainder theorem to obtain $\det(A)$ modulo $(\prod_i p_i)$ and therefore $\det(A)$.
We refer the reader to \cite{Biasse2014} for details on the required size of $(\prod_i p_i)$.
\par
Now consider the general case with $A \in K^{n\times m}$, $n \geq m$.
In \cite[Definition 1.4.9]{Cohen2000} the notion of \textit{minor ideals} of pseudomatrices is introduced, which is a natural extension of minors to pseudomatrices (instead of extracting only rows and columns one also has to take care of the coefficient ideals). Moreover it is shown that the \textit{determinantal ideal} $\mathfrak d \subseteq \mathcal O$ of $\mathcal P$, which is defined to be the sum of all $m \times m$ minor ideals of $\mathcal P$, satisfies $\mathfrak d \mathcal O^m \subseteq M$.
Note that since in general there are just too many minor ideals (as in the case of minors of matrices), in order to find an ideal $\mathfrak m$ with $\mathfrak m \mathcal O^m \subseteq M$ it is sufficient to compute only \textit{one} non-zero minor ideal (which exists since $M$ has rank $m$).

\section{Splitting the modulus}\label{sec:splitting}

In order to speed up computations, we would like, if possible to split
the modulus, the idea being that if $\mathfrak m = \mathfrak a\mathfrak b$, then, by the Chinese remainder theorem, 
$(\mathcal O/\mathfrak m) = (\mathcal O/\mathfrak a) \times (\mathcal O/\mathfrak b)$
and thus ``everything'' modulo $\mathfrak m$ can be done more efficiently
by computing in $(\mathcal O/\mathfrak a)$ and $(\mathcal O/\mathfrak b)$.
If we allow for a complete factorization, we of course achieve
$(\mathcal O/\mathfrak m) = \prod_\mathfrak p (\mathcal O/\mathfrak p^{v_\mathfrak p(\mathfrak m)})$,
however, for general $\mathfrak m$, a factorization is prohibitively expensive.
We observe that the complete factorization would result in the best complexity!

Furthermore, for any prime $\mathfrak p$ of degree one we have
\[ (\mathcal O/\mathfrak p^k) \cong \mathbf Z/p^k\mathbf Z \]
for $p$ the rational prime with $\mathfrak p \cap \mathbf Z = (p)$.
Again, the Chinese remainder
theorem, this time for $\mathbf Z$, allows us to combine any
degree one prime ideals with distinct underlying rational primes
into one, thus obtaining:
\[ (\mathcal O/\mathfrak m) \cong (\mathbf Z/m\mathbf Z) \times (\mathcal O/\mathfrak m') \]
with some potentially much smaller ideal $\mathfrak m'$. 
Once such a decomposition is obtained, much faster algorithms for $\mathbf Z/m\mathbf Z$ can be applied for hopefully a large part of the ring.

Unfortunately, without the use of factorization such a complete splitting
is difficult to achieve.
We propose the following simple algorithm which
is aimed at computing a large portion of the ``degree one part'' while still being
fast.

\begin{algorithm}[($\mathbf Z$-split)]\label{alg:zsplit}
Let $\mathfrak m$ be an integral ideal. The following steps will
produce coprime integral ideals $\mathfrak a$, $\mathfrak b$ with $\mathfrak a\mathfrak b=\mathfrak m$ and a rational integer $m \in \Z$ such that
$(\mathcal O/\mathfrak a) \cong \mathbf Z/m\mathbf Z$
\begin{enumerate}
  \item Let $m = \min( \mathbf Z_{\geq 1} \cap \mathfrak m)$ and $b = \inorm(\mathfrak m)/m$.
\item repeat
\item \quad compute $g = \gcd(m, b)$, $m = m/g$ and $b = b^2 \bmod m$,
\item until $g=1$.
\item Compute $\mathfrak a = m\mathcal O + \mathfrak m$ and $\mathfrak b = (\inorm(\mathfrak m)/m) \mathcal O + \mathfrak m$.
\item return $\mathfrak a,\mathfrak b$.
\end{enumerate}
\end{algorithm}

Note that this algorithm will not necessarily find a maximal
ideal $\mathfrak a \mid \mathfrak m$ such that $(\mathcal O/\mathfrak a) \cong \mathbf Z/m\mathbf Z$ and $\mathfrak a$, $\mathfrak m\mathfrak a^{-1}$ are coprime:
Let $\mathfrak m = \mathfrak p_1\mathfrak p_2 \mathfrak q_1\mathfrak q_2$
where $\mathfrak p_i$, $\mathfrak q_i$ are primes of degree one lying above
distinct rational primes $p$ and $q$ respectively. Then
$\min(\mathfrak m) = pq$ and $\inorm(\mathfrak m) = p^2q^2$, so the algorithm
will terminate with $\mathfrak a = \mathcal O$. 
However, $\mathfrak a = \mathfrak p_1\mathfrak q_1$ would be a correct result---but we need to actually factorize $\mathfrak m$ to find this decomposition.

\begin{proof}[of correctness]
  For any integral ideal $\mathfrak a$ the minimum $\min(\mathfrak a) = \min(\mathbb \Z_{\geq 1} \cap \mathfrak a)$ is equal to $\exp(\mathcal O/\mathfrak a)$ (the exponent of the abelian group $(\mathcal O/\mathfrak a)$):
  Clearly, $\min(\mathfrak a)\in \mathfrak a$ and $\ord (1) = \min\mathfrak a$
  where $\ord$ is the order of the element. Thus if $\inorm(\mathfrak a) = \lvert \mathcal O/\mathfrak a \rvert = \min(\mathfrak a)$, then $(\mathcal O/\mathfrak a) \cong \mathbf Z/\!\min(\mathfrak a)\mathbf Z$, generated by $1$.

From the decomposition above we see that if $\inorm(\mathfrak a) \ne \min( \mathfrak
a)$, then we either have a prime $\mathfrak q$ dividing $\mathfrak a$ of degree
greater then one, we have at least two distinct prime ideals $\mathfrak q_i \mid \mathfrak a$
($i=1,2$) lying above the same rational prime or we have for some ramified 
prime $\mathfrak q$ with $\mathfrak q^2 \mid \mathfrak a$: In the first case $(\mathcal
O/\mathfrak q, +)$ is a non-cyclic group, in the second case
we have a product of $2$ cyclic groups with non-coprime orders
while in the last case clearly $\min(\mathfrak q) = \min(\mathfrak q^2)$, but
$\inorm(\mathfrak q) \ne \inorm(\mathfrak q^2)$.
In all other cases $\mathfrak a$ is composed of powers of degree one 
prime ideals over distinct rational primes as well as ramified primes with exponent 1.

In the algorithm $b$ initially contains all rational primes $q$ such that
either $\mathfrak q \mid q$ for some prime of degree greater then one, $\mathfrak q_i \mid q$ with $i=1,2$
or $\mathfrak q^2 \mid \mathfrak a$ for some ramified prime $\mathfrak q \mid q$.
During the loop, we remove all those rational primes from $m$ and in the final step we then split $\mathfrak m$ accordingly. The squaring of $b$ ensures that the total time is polynomially bounded.
\end{proof}

Let $\mathfrak m = \mathfrak a \mathfrak b$ be the splitting obtained by this algorithm.
Experimentally, we have $\inorm(\mathfrak b) \ll \inorm(\mathfrak a)$, in fact frequently, $\inorm(\mathfrak b)=1$,
thus the effort to compute a pseudo-HNF over a number field is mostly independent of
it's degree and depends almost \textit{only} on the dimension of the matrix.

We note that the CRT techniques for the HNF are also used to derive (expected)
polynomial complexity in the presence of lots of small prime ideals: The
runtime depends on $p_\mathfrak m$ and $p_\mathfrak m$ is mainly determined by the norms
of the small prime ideals. Thus we use the approach of Belabas (see~\cite[Section 6]{Belabas2004}) to split the 
modulus into small primes, where we can directly use his (deterministic)
linear algebra approach to find uniformizing elements and thus work in the 
completions. For the (large) remainder term, we use our randomized methods.

\section{Computations}

We have implemented both the Euclidean structure and the improved pseudo-HNF computation in the computer algebra system \textsc{Magma}~\cite{Bosma1997}.
To illustrate the efficiency of our techniques, we computed 
pseudo-HNFs for random matrices over a range of
fields. In particular, we used $K = \mathbf Q[t]/(t^d-10)$
for $d=2$, $4$, $8$, and generated matrices of
dimensions $n$ up to $300$, depending on $d$. More specifically,
starting at $k=1$, we computed for two random matrices $A$ of dimension $n = 10\cdot k$ a pseudo-HNF of the pseudomatrix $((\mathcal O)_{1 \leq i \leq n},A)$ both using 
our method and \textsc{Magma}'s implementation of Cohen's algorithm (available through the command \textsf{HermiteForm}) until a 
single computation took more than one hour.
By random matrices we mean matrices over $\mathcal O$, where the coefficients (with respect to a fixed integral basis) of the matrix entries are chosen uniformly in $\{ -2^B, \dotsc, 2^B \}$ for the times $t_1$, $t_2$ and rounded normally distributed with mean $0$ and variance $2^{2B}$ for the times $g_1$, $g_2$.
Table~\ref{tab:comp} shows the results for different choices of parameters $d$, $n$ and $B$, where $t_1$ (resp. $g_1$) denotes the running time (in seconds) using Algorithm~\ref{alg:sum} and $t_2$ (resp. $g_2$) the running time (in seconds) using \textsc{Magma}'s implementation of Cohen's algorithm.
We briefly note, that the longer running times for the normal distributed
matrix entries are a consequence of them being larger: By Hadamard's inequality,
the size of the determinant depends mainly on the largest entry in each row or column respectively. Using normal distributed entries, this maximum value will usually be larger
than $2^B$, which is reflected in the runtime.

\begin{table}[ht]\vspace*{-3ex}
\caption{Algorithm~\ref{alg:sum} versus \textsc{Magma}'s \textsf{HermiteForm}}
\label{tab:comp}
 \begin{tabular}{cccrrrrrr}
   \hline
    $d$ & $B$ & $n$ & $t_1$  & $t_2$ & $t_2/t_1$ & $g_1$ & $g_2$ & $g_2/g_1$ \\
    \hline
    2 & 10 & 10   & 0.095   & 0.020    & 0.210 & 0.030   & 0.010    & 0.333 \\
      &    & 20   & 0.130   & 0.065    & 0.500 & 0.335   & 0.080    & 0.238 \\
      &    & 30   & 0.375   & 0.210    & 0.560 & 0.465   & 0.155    & 0.333 \\
      &    & 40   & 0.325   & 0.300    & 0.923 & 0.405   & 0.360    & 0.888 \\
      &    & 200  & 107.715 & 143.975  & 1.336 & 128.335 & 165.475  & 1.289 \\ 
      &    & 300  & 580.370 & 1031.430 & 1.777 & 842.675 & 1210.775 & 1.436 \\
      \hline
    2 & 100 & 10  & 0.075   & 0.155    & 2.066 & 0.055   & 0.090    & 1.636 \\
      &     & 20  & 0.380   & 0.655    & 1.723 & 0.400   & 0.740    & 1.850 \\
      &     & 30  & 1.245   & 2.490    & 2.000 & 1.455   & 2.890    & 1.986 \\
      &     & 40  & 3.265   & 6.985    & 2.139 & 3.155   & 10.630   & 3.369 \\
      &     & 80  & 47.945  & 107.115  & 2.234 & 51.495  & 107.320  & 2.084 \\
      &     & 140 & 549.080 & 1194.445 & 2.175 & 540.660 & 1008.665 & 1.865 \\
      \hline
    4 & 10  & 10  & 0.080   & 0.055    & 0.687 & 0.055   & 0.085    & 1.545 \\
      &     & 20  & 0.260   & 0.390    & 1.500 & 0.195   & 0.385    & 1.974 \\
      &     & 30  & 0.525   & 1.040    & 1.980 & 0.640   & 1.325    & 2.070 \\
      &     & 40  & 1.955   & 3.080    & 1.575 & 0.945   & 3.440    & 3.640 \\
      &     & 80  & 10.080  & 37.970   & 3.515 & 12.165  & 48.505   & 3.987 \\
      &     & 140 & 77.640  & 346.315  & 4.460 & 107.005 & 402.735  & 3.763 \\
      \hline
    8 & 10  & 10  & 0.290   & 0.850    & 2.931 & 0.160   & 0.660    & 4.125 \\
      &     & 20  & 0.620   & 5.345    & 8.620 & 1.445   & 6.955    & 4.813 \\
      &     & 30  & 1.605   & 26.470   & 16.492 & 1.785   & 33.190   & 18.593 \\
      &     & 40  & 5.675   & 57.535   & 10.138 & 7.355   & 96.797   & 13.160 \\
      &     & 80  & 48.445  & 746.120  & 15.401 & 44.720  & 917.765  & 20.522 
  \end{tabular}
\end{table}

\section{Conclusions}

In the preceding sections, we presented a suite of new algorithms to
explicitly utilize the Euclidean structure of quotients 
of rings of integers. The power of those ideas was demonstrated via
a new, probabilistic, modular algorithm to compute normal forms of modules over
rings of integers. The resulting algorithm is both faster and conceptually
simpler as it does not need to work with the pseudobases and the coefficient
ideals.

Our new lifting algorithm to obtain a non-modular Hermite form is even
in the case of $\Z$-modules new and conceptually simpler than 
the usual lifting algorithm: We do not need to perform any elimination
steps in characteristic $0$, all is done through the adapted Howell normal form on the
modular ``side''.

While we did not do a complete bit complexity analysis for the pseudo Hermite normal form algorithm (Algorithm~\ref{alg:sum}),
it is clear that the method presented has polynomial expected complexity:
The modular algorithms all use an expected polynomial number of operations on 
elements of a bounded size and the lifting steps are easily realized
using linear algebra over $\Z$. The comparison with 
Cohen's algorithm on theoretical grounds is difficult as his algorithm is
not analyzed and conjectured to have exponential runtime due to
intermediate coefficient swell. A suitably modified modular
version was proven to be polynomial time in \cite{Biasse2012}, but
while the complexity in the module dimension is the same, the complexity
in the field degree is far worse there due to the expensive ideal operations
in particular the lattice basis reduction to keep the ideals bounded in size.

Future work will try to find faster and deterministic algorithms.

\bibliographystyle{plain}
\bibliography{paper_quotient}

\affiliationone{
   Claus Fieker and Tommy Hofmann\\
   Fachbereich Mathematik \\
   Technische Universit\"{a}t Kaiserslautern \\
   Postfach 3049 \\
   67653 Kaiserslautern \\
   Germany
   \email{fieker@mathematik.uni-kl.de\\
   thofmann@mathematik.uni-kl.de}}

\end{document}